\documentclass[11pt,a4paper]{amsart}
\usepackage{amsmath,amsthm,amssymb,latexsym}
\usepackage{graphicx}

\newcommand{\bn}{{\mathbb{N}}}
\newcommand{\br}{{\mathbb{R}}}

\newcommand{\bz}{{\mathbb{Z}}}
\newcommand{\bc}{{\mathbb{C}}}

\newcommand{\cd}{{\mathcal{D}}}

\newcommand{\cp}{{\mathcal{P}}}

\renewcommand{\a}{\alpha}
\renewcommand{\b}{\beta}
\renewcommand{\l}{\lambda}

\newcommand{\s}{\sigma}

\renewcommand{\d}{\delta}

\renewcommand{\o}{\omega}

\newcommand{\g}{\gamma}

\newcommand{\ep}{\varepsilon}

\newcommand{\lp}{\left(}
\newcommand{\rp}{\right)}
\newcommand{\lt}{\left\{}
\newcommand{\rt}{\right\}}

\DeclareMathOperator{\Sp}{Sp}
\DeclareMathOperator{\tr}{tr}

\allowdisplaybreaks
\numberwithin{equation}{section}

\newtheorem{theorem}{Theorem}[section]

\newtheorem{lemma}[theorem]{Lemma}

\newtheorem{proposition}[theorem]{Proposition}

\theoremstyle{definition}

\newtheorem{remark}[theorem]{Remark}

\begin{document}

\title[Borg--Hochstadt theorem]
{On stability in the Borg--Hochstadt theorem for periodic Jacobi matrices}
\author[L. Golinskii]{L. Golinskii}

\address{B. Verkin Institute for Low Temperature Physics and
Engineering, 47 Science ave., Kharkiv 61103, Ukraine}
\email{golinskii@ilt.kharkov.ua}

\date{\today}

\keywords{periodic Jacobi matrices, spectrum, bands and gaps, Hill discriminant, trace formulae}
\subjclass[2010]{47B36, 47B39}

\maketitle

\begin{abstract}
A result of Borg--Hochstadt in the theory of periodic Jacobi matrices states that such a matrix has constant diagonals as long as all
gaps in its spectrum are closed (have zero length). We suggest a quantitative version of this result by proving the two-sided bounds
between oscillations of the matrix entries along the diagonals and the length of the maximal gap in the spectrum.
\end{abstract}

\vspace{0.5cm}
\begin{center}
{\large\it To V. A. Marchenko on occasion of his 95-th anniversary}
\end{center}

\vspace{0.5cm}
\section*{Introduction}
\label{s0}

In his recent paper \cite{ma14} V. A. Marchenko reverts to the classical subject of periodic Jacobi matrices. He gives an intrinsic description of
polynomials with all their $\pm2$-points real and so obtains parametrization of the Hill discriminants of such matrices. His argument
is straightforward and makes no appeal to conformal mappings.

The Hill discriminants play a crucial role in the problem related to periodic Jacobi matrices we address in this note. Recall that Jacobi matrices are
two-sided, infinite, three-diagonal matrices of the form
\begin{equation}\label{defjac}
J =
\begin{bmatrix}
\ddots & \ddots & \ddots &  & \\
 & a_{-1} & b_0 & a_0 &  & \\
 & & a_0 & b_1 & a_1 & & \\
 & &  & a_1 & b_2 & a_2 & \\
 & &  & & \ddots & \ddots & \ddots
\end{bmatrix}, \quad b_j\in\br, \quad a_j>0.
\end{equation}
$J$ is said to be {\it periodic of period} $p\in\bn:=\{1,2,\ldots\}$, if
$$ a_{j+p}=a_j, \quad b_{j+p}=b_j, \quad j\in\bz:=\{0,\pm1,\pm2,\ldots\}. $$

A $p$-periodic Jacobi matrix $J$ \eqref{defjac} generates in an obvious way a bounded, self-adjoint, linear operator $J$ on the Hilbert space $\ell^2(\bz)$.
Its spectrum $\s(J)$ is known to have a banded structure, i.e., it is composed of $p$ {\it spectral bands} (closed intervals)
\begin{equation}\label{spect}
\s(J)=\bigcup_{j=0}^{p-1} [\mu_j^+, \mu_{j+1}^-], \ \ \mu_0^+<\mu_1^-\le\mu_1^+<\ldots <\mu_{p-1}^-\le\mu_{p-1}^+< \mu_{p}^-,
\end{equation}
some of them can merge. A convex hall of the spectrum (the least interval that contains the whole spectrum) is $L=[\mu_0^+, \mu_p^-]$.

The bands are interspersed with (interior) {\it gaps}
\begin{equation}\label{gap}
\g_j:=(\mu_j^-,\mu_j^+), \quad j=1,2,\ldots,p-1, \quad \mu_j^- \le \mu_j^+,
\end{equation}
of the length $|\g_j|=\mu_j^+-\mu_j^-$, and $\mu_j^- = \mu_j^+$ means that the gap is closed (the adjacent bands merge).
Yet it seems advisable viewing a closed gap as an actual gap
(of zero length) rather than dealing with two merging bands as a single one. We stick to this viewpoint consistently throughout the paper.
We observe the situation, when the closed gaps arise, in the simplest example of constant Jacobi matrices $J_0$ with
$a_j=a_0$, $b_j=b_0$, $j\in\bz$. Now all gaps are closed, and the spectrum $\s(J_0)=[b_0-2a_0, b_0+2a_0]$ is a single interval.
Following the above convention, we can (and will) view such matrices as periodic of period $p\in\bn$.

The well-known result of Borg--Hochstadt \cite{hoch75, hoch84}, (see \cite[Theorem 5.4.21]{simsz}), states, that the converse is also true. Precisely,
a periodic Jacobi matrix $J$ with all gaps closed is constant, $J=J_0$. The problem we address here is stability (or a quantitative version)
of this result. Specifically, we show that for periodic Jacobi  matrices with the ``small'' variation of parameters a's and b's, the gaps in
their spectra are ``small'', and vice versa.

To be more formal, given a bounded sequence $c=\{c_j\}_{j\in\bz}$ of real numbers, its {\it variation} is defined by
\begin{equation}\label{var}
\o_c:=\sup_{i,j\in\bz}(c_i-c_j).
\end{equation}
Let $\g$ be a maximal gap in the spectrum of $J$, so
\begin{equation}\label{maxgap}
|\g|=\max_{1\le j\le p-1} |\g_j|.
\end{equation}

Here is the main result of the paper.

\begin{theorem}\label{main}
Let $J$ be a periodic Jacobi matrix $\eqref{defjac}$ of period $p$. Then
\begin{equation}\label{upbo}
\o_b\le \frac{p(p-1)}2\,|\g|, \qquad \qquad \o_a\le p^2\sqrt{p}\,|\g|.
\end{equation}
Furthermore,
\begin{equation}\label{lobo}
\o_b+\o_a\ge \frac{|\g|}4\,.
\end{equation}
\end{theorem}
In particular, $|\g|=0$ (all gaps are closed) if and only if $\o_a=\o_b=0$ ($J$ is a constant Jacobi matrix), so the result of Borg--Hochstadt follows.
In the case of periodic Jacobi matrices the supremum in the definition \eqref{var} is obviously attained for the matrix entries take only finite number
of values.

A periodic Jacobi matrix is said to be {\it normalized} if
\begin{equation}\label{nor}
\sum_{j=1}^p b_j=0.
\end{equation}
Clearly, each Jacobi matrix can be normalized by adding an appropriate constant to all b's. Under such transformation (shift) neither the variation of
the entries along diagonals, nor the length of the maximal gap alters, so throughout the paper we assume \eqref{nor} to hold.

As in \cite{ma14}, the argument is by and large elementary, and relies upon some basic properties of the
Hill discriminant with regard to the spectrum $\s(J)$, see \cite[Chapter 7]{tesc}, \cite[Chapter 5]{simsz}.

A key ingredient of the proof is the extremal problem ``${\rm\grave{a}}$ la'' the Chebyshev Alternance Theorem,
suggested by Korotyaev--Kutsenko \cite[Lemma 2.2]{koku09}.

Let $x=(x_1,\ldots,x_n)\in\br^n$, $n\ge2$, and put
$$ X(\l)=X(\l,x):=\prod_{j=1}^n (\l-x_j). $$
A set $\cp_n(c)\subset\br^n$ is defined by imposing the following conditions: \newline
$x=(x_1,\ldots,x_n)\in\cp_n(c)$ if and only if
\begin{enumerate}
  \item $x_1+\ldots+x_n=0$;
  \item $|X(\l)|\le c, \quad \l\in[\min_j x_j, \max_j x_j]$.
\end{enumerate}
The result of Korotyaev--Kutsenko states that
\begin{equation}\label{koku}
\sup\{ \|x\|^2=x_1^2+\ldots+x_n^2:\  x\in\cp_n(c)\}=2n\lp\frac{c}2\rp^{2/n}.
\end{equation}

It is instructive to paraphrase this result in a quantitative form. An algebraic polynomial $P$ with the roots $\{x_j\}$ is {\it balanced}, if
$x_1+\ldots+x_n=~0$.

{\bf Theorem KK}.
{\it
Let
$$ P(x)=\prod_{j=1}^n (x-x_j)=x^n+\s_2 x^{n-2}+\ldots, \quad x_1\le\ldots\le x_n, $$
be a balanced polynomial with the real roots $\{x_j\}$. Then
\begin{equation}\label{quankoku}
\sum_{j=1}^n x_j^2\le 2n\lp\frac{\|P\|}2\rp^{2/n}\,, \qquad \|P\|:=\max_{x\in [x_1,x_n]}|P(x)|.
\end{equation}
}

\smallskip

The work on this note was inspired to a large extent by a recent manuscript \cite{kuku}, wherein the authors examine the notion of
$\ep$-pseudospectrum in Jacobi matrices setting. As all operators there are self-adjoint, the $\ep$-pseudospectrum of an operator $J$ agrees with
the $\ep$-neighborhood of the actual spectrum $\s(J)$. That is why the $\ep$-pseudospectrum is connected if and only if $|\g|\le 2\ep$. So, up to
distinction in terminology, \cite{kuku} contains the upper bound for $\o_b$ and the lower bound \eqref{lobo} (even with better constants).
The main contribution of this note is a new upper bound for $\o_a$ in \eqref{upbo}, which answers a question posed in \cite[Remark 3.3]{kuku}.

We illustrate our results on the example of $4$-periodic Jacobi matrices.

{\bf Acknowledgement}. The author thanks I. Egorova and A. Kutsenko for valuable discussions.

\section{Proof of the main result: upper bounds}
\label{s1}

We recall some rudiments of the theory of Jacobi matrices \cite{tesc, simsz}.

The basic recurrence relation for Jacobi matrices \eqref{defjac} is of the form
\begin{equation}\label{recur}
Jy=\l y \ \sim \ a_{n-1}y_{n-1}+b_ny_n+a_ny_{n+1}=\l y_n, \quad \l\in\bc, \quad n\in\bz,
\end{equation}
and a pair of special solutions of \eqref{recur}, $\{c_n(\cdot,m)\}_{n\in\bz}$ and $\{s_n(\cdot,m)\}_{n\in\bz}$, $m\in\bz$, with the initial data
\begin{equation}\label{spesol}
\begin{split}
s_{m-1}(\l,m) &=0,  \qquad \quad s_m(\l,m)=1, \\
c_{m-1}(\l,m) &=-1, \qquad c_m(\l,m)=0,
\end{split}
\end{equation}
is of particular concern. It is clear that $s_{m+k}(\cdot,m)$ ($c_{m+k}(\cdot,m)$) is the algebraic polynomial of degree $k$ ($k-1$), respectively.

The polynomial $s_{m+p-1}$ is crucial for the theory of periodic Jacobi matrices. Its zeros $\{\xi_j^{(m)}\}_{j=1}^{p-1}$ are
known to be simple and real, and
\begin{equation*}
\begin{split}
a^p s_{m+p-1}(\l,m) &=\prod_{j=1}^{p-1}\lp\l-\xi_j^{(m)}\rp=\l^{p-1}-\l^{p-2}\,\sum_{j=0}^{p-2}b_{m+j}+\ldots, \\
a &:=(a_1\ldots a_p)^{1/p}.
\end{split}
\end{equation*}
In view of normalization \eqref{nor} and periodicity,
$$ \sum_{j=1}^{p-1}\xi_j^{(m)}=\sum_{j=0}^{p-2}b_{m+j}=-b_{m+p-1}=-b_{m-1}, $$
so the following trace formula is valid
\begin{equation}\label{tr1}
b_m-b_{m-1}=\sum_{j=1}^{p-1} \lp\xi_j^{(m)}-\xi_j^{(m+1)}\rp.
\end{equation}

On the other hand, for each $m\in\bz$ the zero $\xi_j^{(m)}\in\bar\g_j=[\mu_j^-,\mu_j^+]$, see, e.g., \cite[formula (7.50)]{tesc},
\cite[Theorem 5.4.16]{simsz}, so
$$ |b_m-b_{m-1}|\le \sum_{j=1}^{p-1} \left|\xi_j^{(m)}-\xi_j^{(m+1)}\right|\le\sum_{j=1}^{p-1}|\g_j|\le (p-1)|\g|, $$
and the upper bound for $\o_b$ in \eqref{upbo} follows.

\smallskip

The upper bound for $\o_a$ is much more intricate. Now the Hill discriminant $\cd$ appears on the central stage (see \cite[Section 5.4]{simsz}
for a detailed account of discriminants and their properties). This object arises in two guises:
as a difference of the special solutions of \eqref{recur} (or, equivalently, the trace of the monodromy matrix)
\begin{equation*}
\cd(\l)=s_{m+p}(\l,m)-c_{m+p-1}(\l,m),
\end{equation*}
and as a characteristic polynomial of a certain Hermitian matrix
\begin{equation*}
a^p\,\cd(\l)=\det(\l-\Phi_m)=\prod_{j=1}^p(\l-d_j),
\end{equation*}
where
\begin{equation}\label{matsym}
\Phi_m=
\begin{bmatrix}
b_m & a_m &  &  &  & a_{m+p-1}i \\
a_m & b_{m+1} & a_{m+1} &  &  & \\
& a_{m+1} & b_{m+2} & a_{m+2} & & \\
&  & \ddots & \ddots & \ddots & \\
&  &  & a_{m+p-2} & b_{m+p-2} & a_{m+p-2} \\
-a_{m+p-1}i & & & & a_{m+p-1} & b_{m+p-1}
\end{bmatrix}.
\end{equation}
As a matter of fact, $\cd$ does not depend on $m$, and neither do its zeros $\{d_j\}$. By normalization \eqref{nor} and \eqref{matsym},
$\cd$ is a balanced polynomial of degree $p$.

The zeros $\{d_j\}_{j=1}^p$ of $\cd$ (or, equivalently, the eigenvalues of $\Phi_m$) are known to be real and simple.
Evaluating of the trace of $\Phi_m^2$ in two ways provides, in view of periodicity, another trace formula
\begin{equation}\label{trfo1}
\tr\Phi_m^2=\sum_{j=1}^p d_j^2=\sum_{j=1}^p (b_j^2+2a_j^2).
\end{equation}

There is a tight relation between the Hill discriminant and the spectrum of the underlying Jacobi matrix $J$. Precisely,
$$ \s(J)=\cd^{(-1)}[-2,2], \qquad |\cd(\mu_j^{\pm})|=2, $$
so $\s(J)$ is the inverse image of the interval $[-2,2]$, and $|\cd|=2$ at all endpoints of the spectrum. In other words, the set
$\{\mu_j^{\pm}\}_j$ constitutes a full collection of $\pm2$-points of the Hill discriminant.

Denote by $2M$ the sup-norm of the discriminant over the convex hall of the spectrum
$$ 2M=\|\cd\|=\|\cd\|_{C(L)}, \qquad L=[\mu_0^+, \mu_p^-]. $$
Then, $M\ge1$, and by the Borg--Hochstadt theorem, $M=1$ if and only if $|\g|=0$, so $J$ is the constant matrix. We exclude this
case in what follows, so $M>1$. The discriminant attains this extremum at some critical point $\nu_k\in\g_k=(\mu_k^-,\mu_k^+)$.
By the Taylor formula, there is $\xi\in[\mu_k^-,\mu_k^+]$ so that for $\mu_k^-\le t\le\mu_k^+$ the equality
$$ \cd(t)=\cd(\nu_k)+\cd'(\nu_k)\,(t-\nu_k)+\frac{\cd''(\xi)}2\,(t-\nu_k)^2=\cd(\nu_k)+\frac{\cd''(\xi)}2\,(t-\nu_k)^2 $$
holds. Putting $t=\mu_k^+$ provides
$$ 2(M-1)=\frac{|\cd''(\xi)|}2\,(\mu_k^+-\nu_k)^2, $$
or
$$ (\mu_k^+-\nu_k)^2=\frac{4(M-1)}{|\cd''(\xi)|}\ge \frac{4(M-1)}{\|\cd''\|}\,. $$
By A. Markov's inequality for the interval $L$ \cite[Theorem 5.1.8]{boer},
$$ \|\cd''\|\le \frac{4p^4\,M}{|L|^2}\,, $$
and we end up with the lower bound for the length of the maximal gap
\begin{equation}\label{low1}
|\g|^2\ge (\mu_k^+-\nu_k)^2\ge\frac{|L|^2}{p^4}\cdot\frac{M-1}{M} \ \Rightarrow \ |\g|\ge\frac{|L|}{p^2}\cdot\lp\frac{M-1}{M}\rp^{1/2}.
\end{equation}

We distinguish two cases.

1. The case when $M\ge2$ is simple. We show that not only the variation of $a_j$'s, but the numbers themselves are ``small''. Indeed,
it follows now from \eqref{low1} that
\begin{equation}\label{low2}
|\g|\ge\frac{|L|}{\sqrt2 p^2}\,.
\end{equation}
Note that for an algebraic polynomial of degree $p\ge2$ the sum of its $\a$-points does not depend on $\a$. In particular, for a balanced
polynomial the sum of its $\a$-points is zero. Hence, by \eqref{nor},
$$ \mu_0^+ + \mu_p^-+\sum_{j=1}^{p-1} (\mu_j^-+\mu_j^+)=2\sum_{j=1}^p b_j=0, $$
(the left-hand side being the sum of all $\pm2$-points of $\cd$), so $\mu_0^+<0<\mu_p^-$. Since $d_j\in\s(J)$, we have
$$ |d_j|\le |L|, \qquad j=1,2,\ldots,p, $$
and the trace formula \eqref{trfo1} implies
$$  \sum_{j=1}^p (b_j^2+2a_j^2)=\sum_{j=1}^p d_j^2<p\,|L|^2. $$
In view of \eqref{low2}, we come to the bound
\begin{equation}\label{avar}
\max_j a_j<\sqrt{\frac{p}2}\,|L|\le p^2\sqrt{p}\,|\g|,
\end{equation}
as claimed.

2. Now for the harder part $1<M<2$. The extremal problem of Korotyaev--Kutsenko comes into play here. We apply Theorem~KK to
the balanced polynomial $P=a^p\cd$, to obtain
$$ \sum_{j=1}^p d_j^2\le 2p\lp \frac12\,\|P\|_{C[d_1,d_p]}\rp^{2/p}\le 2p a^2 M^{2/p}. $$
By the trace formula \eqref{trfo1},
$$ \frac1p\,\sum_{j=1}^p a_j^2\le a^2 M^{2/p}, $$
and the arithmetic-geometric means (AGM for short) inequality yields
\begin{equation}\label{upbo2}
\frac1p\,\sum_{j=1}^p a_j^2-\lp\frac1p\,\sum_{j=1}^p a_j\rp^2\le \frac1p\,\sum_{j=1}^p a_j^2-a^2\le a^2(M^{2/p}-1).
\end{equation}
But the left-hand side of \eqref{upbo2} is
$$ \frac1p\,\sum_{j=1}^p a_j^2-\lp\frac1p\,\sum_{j=1}^p a_j\rp^2=\frac1p\,\sum_{j=1}^p (a_j-s_a)^2, \quad s_a:=\frac1p\,\sum_{j=1}^p a_j. $$
To estimate the right-hand side of \eqref{upbo2} from above, we apply the well-known inequality
$$ u^\tau-1\le \tau(u-1), \qquad u>0, \quad 0<\tau\le1. $$
Thus, we come to the upper bound
$$ \frac1p\,\sum_{j=1}^p (a_j-s_a)^2\le \frac{2a^2}{p}\,(M-1). $$

To proceed further, recall the assumption $1<M<2$. From \eqref{low1} we derive
$$ |\g|\ge\frac{|L|}{\sqrt2 p^2}\,\sqrt{M-1}, \quad M-1\le\frac{2p^4}{|L|^2}\,|\g|^2, $$
and so
$$ \frac1p\,\sum_{j=1}^p (a_j-s_a)^2\le \frac{4a^2p^3}{|L|^2}\,|\g|^2. $$
By \cite[Lemma 2.2]{kokr}, $|L|=\mu_p^- -\mu_0^+\ge 4a$, so, finally,
\begin{equation}\label{upbo3}
\frac1p\,\sum_{j=1}^p (a_j-s_a)^2\le \frac{p^3}4\,|\g|^2, \quad |a_j-s_a|\le \frac{p^2}2\,|\g|
\end{equation}
for each $j=1,2,\ldots,p$. The latter inequality obviously implies  \eqref{upbo} for $\o_a$. This proves the
upper bounds in Theorem \ref{main}.

\section{$4$-periodic Jacobi matrices}

Consider an instructive example of $4$-periodic Jacobi matrices $J$ with $b_j\equiv 0$, $j\in\bz$. The Hill discriminant is now
\begin{equation*}
\begin{split}
\cd(\l) &= s_{m+4}(\l,m)-c_{m+3}(\l,m)=a^{-4}(\l^4-\a\l^2+\b), \\
a^4 &=a_1a_2a_3a_4,  \quad \a:=a_1^2+a_2^2+a_3^2+a_4^2, \quad \b:=(a_1a_3)^2+(a_2a_4)^2.
\end{split}
\end{equation*}
The equations $\cd(\l)\pm2=0$ are biquadratic, and their roots are
\begin{equation*}
\begin{split}
\cd(\l)-2 &=0 \ \Rightarrow \ \l=\{\pm\l_1^-, \pm\l_1^+\}, \\
\cd(\l)+2 &=0 \ \Rightarrow \ \l=\{\pm\l_2^-, \pm\l_2^+\}.
\end{split}
\end{equation*}
The above biquadratic equations can be solved explicitly
\begin{equation*}
\l_1^{\pm} = \sqrt{\frac{\a\pm\sqrt{D^+}}2}\,, \qquad   \l_2^{\pm} = \sqrt{\frac{\a\pm\sqrt{D^-}}2}\,,
\end{equation*}
where
\begin{equation*}
\begin{split}
D^+ &:=\a^2-4(\b-2a^4) =\a^2-4(a_1a_3-a_2a_4)^2 \\ &=\lt (a_1-a_3)^2+(a_2+a_4)^2\rt\, \lt (a_1+a_3)^2+(a_2-a_4)^2\rt>0, \\
D^- &:=\a^2-4(\b+2a^4) = \a^2-4(a_1a_3-a_2a_4)^2 \\ &=\lt (a_1-a_3)^2+(a_2-a_4)^2\rt\, \lt (a_1+a_3)^2+(a_2+a_4)^2\rt>0.
\end{split}
\end{equation*}
By the AGM inequality,
\begin{equation}\label{neq1}
\sqrt{D^{\pm}}\le\a,
\end{equation}
so all roots are real.

The spectrum is composed of four spectral bands symmetric with respect to the origin
\begin{equation*}
\begin{split}
\Sp(J) &=[\l_1^-,\l_2^-]\bigcup [\l_2^+,\l_1^+]\bigcup [-\l_2^-,-\l_1^-]\bigcup [-\l_1^+,-\l_2^+], \\
0 &\le\l_1^-<\l_2^-\le \l_2^+<\l_1^+.
\end{split}
\end{equation*}
The bands are separated with three gaps, the interior one $\g_{int}=(-\l_1^-,\l_1^-)$, and two symmetric exterior ones,
$\pm\g_{ext}$,  $\g_{ext}=(\l_2^-,\l_2^+)$.
In view of \eqref{neq1}, the length of $\g_{ext}$ is bounded from below by
\begin{equation*}
\begin{split}
\sqrt2\,|\g_{ext}| &=\sqrt2(\l_2^+-\l_2^-) =\sqrt{\a+\sqrt{D^-}}-\sqrt{\a-\sqrt{D^-}} \\
&=\frac{2\sqrt{D^-}}{\sqrt{\a+\sqrt{D^-}}+\sqrt{\a-\sqrt{D^-}}}\ge \sqrt{\frac{D^-}{2\a}}\,.
\end{split}
\end{equation*}
 Therefore,
\begin{equation*}
\begin{split}
2|\g_{ext}| &\ge \sqrt{\frac{D^-}{\a}} =\sqrt{(a_1-a_3)^2+(a_2-a_4)^2}\,\sqrt{\frac{(a_1+a_3)^2+(a_2+a_4)^2}{\a}} \\
&\ge \sqrt{(a_1-a_3)^2+(a_2-a_4)^2},
\end{split}
\end{equation*}
(see the definition of $\a$), and so
\begin{equation}\label{osc1}
\max(|a_1-a_3|, |a_2-a_4|)\le 2|\g_{ext}|.
\end{equation}
In particular, if the exterior gaps are closed, then $a_1=a_3$, $a_2=a_4$, and the actual period is $2$.

As far as the interior gap goes, we see that
\begin{equation*}
\frac12\,|\g_{int}|=\l_1^-=\sqrt{\frac{\a-\sqrt{D^+}}2}=\sqrt2\,\frac{|a_1a_3-a_2a_4|}{\sqrt{\a+\sqrt{D^+}}}\ge\frac{|a_1a_3-a_2a_4|}{\sqrt{\a}}\,.
\end{equation*}
Note that if the interior gap is closed, then $a_1a_3=a_2a_4$ (the period may still be $4$).

Assume with no loss of generality, that $a_1+a_2\ge a_3+a_4$. Since
$$ 1\le\frac{a_1+a_2+a_3+a_4}{\sqrt{\a}}\le 2, $$
we come to the bounds
\begin{equation}\label{osc2}
\frac12\le\frac{a_1+a_2+a_3+a_4}{2\sqrt{\a}}\le\frac{a_1+a_2}{\sqrt{\a}}\le\frac{a_1+a_2+a_3+a_4}{\sqrt{\a}}\le 2.
\end{equation}
By \eqref{osc2},
$$ \frac{|a_1-a_2|}2\le \frac{a_1+a_2}{\sqrt{\a}}\,|a_1-a_2|=\frac{|a_1^2-a_2^2|}{\sqrt{\a}}. $$
On the other hand, we write
$$ a_1^2-a_2^2=(a_1a_3-a_2a_4)+a_1(a_1-a_3)+a_2(a_4-a_2), $$
to obtain, in view of \eqref{osc1}, \eqref{osc2},
\begin{equation*}
\begin{split}
\frac{|a_1^2-a_2^2|}{\sqrt{\a}} &\le \frac{|a_1a_3-a_2a_4|}{\sqrt{\a}}+
\frac{a_1}{\sqrt{\a}}\,|a_1-a_3|+\frac{a_2}{\sqrt{\a}}\,|a_2-a_4| \\
&\le \frac12\,|\g_{int}|+2|\g_{ext}|\,\frac{a_1+a_2}{\sqrt{\a}}\le\frac12\,|\g_{int}|+4|\g_{ext}|.
\end{split}
\end{equation*}
Hence,
$$ |a_1-a_2|\le |\g_{int}|+8|\g_{ext}|\le 9|\g|, \quad |\g|:=\max(|\g_{int}|,|\g_{ext}| ) $$
is the length of the maximal gap in the spectrum.

Next, by \eqref{osc1},
$$ |a_3-a_4|\le |a_3-a_1|+|a_1-a_2|+|a_2-a_4|\le |\g_{int}|+12|\g_{ext}|, $$
and we end up with the following bound for the variation of the a's diagonal
\begin{equation}\label{osc4}
\o_a\le |\g_{int}|+12|\g_{ext}|\le 13|\g|.
\end{equation}

The opposite inequalities can be proved along the same line of reasoning. Indeed,
\begin{equation*}
\begin{split}
|\g_{int}| &=2\sqrt2\,\frac{|a_1a_3-a_2a_4|}{\sqrt{\a+\sqrt{D^+}}}\le\frac{2\sqrt2}{\sqrt{\a}}\,\lp |a_3-a_4|a_1+|a_1-a_2|a_4\rp \\
&\le 2\sqrt2 \o_a\,\frac{a_1+a_4}{\sqrt{\a}}\le 4\sqrt2\,\o_a,
\end{split}
\end{equation*}
and
\begin{equation*}
\begin{split}
|\g_{ext}| &=\frac{\sqrt{2D^-}}{\sqrt{\a+\sqrt{D^-}}+\sqrt{\a-\sqrt{D^-}}}\le\sqrt{\frac{D^-}{\a}} \\
&\le \sqrt2\o_a\,\sqrt{\frac{(a_1+a_3)^2+(a_2+a_4)^2}{\a}}\le 2\o_a.
\end{split}
\end{equation*}

\section{Proof of the main result: lower bound}
\label{s3}

We suggest here a proof of the lower bound \eqref{lobo}, alternative to one in \cite[Theorem 3.2]{kuku}. It has nothing to do
with periodicity and applies to arbitrary bounded Jacobi matrices. The argument is based on some general facts from the perturbation
theory (perturbation of the spectra).

We say that $(\l,\l')$ is a gap in the spectrum of a bounded, self-adjoint operator $T$ on a Hilbert space, if
$$ (\l,\l')\bigcap\s(T)=\emptyset, \qquad \l,\l'\in\s(T). $$

The result below is likely to be well known. We provide its proof for the sake of completeness.

\begin{lemma}\label{perspe}
Let $T_0$ be a bounded, self-adjoint operator on a Hilbert space with the spectrum $\s(T_0)=[u_0,v_0]$, a single interval.
Let $T$ be a bounded, self-adjoint operator so that $\|T-T_0\|\le\d$. Then the length of each gap $(\mu^-,\mu^+)$ in $\s(T)$ does not
exceed $2\d$.
\end{lemma}
\begin{proof}
Denote by $[u,v]$ the convex hall of $\s(T)$, i.e., the least closed interval, which contains $\s(T)$. By the hypothesis on the perturbation,
we have
$$ u_0\le u+\d, \qquad v_0\ge v-\d. $$
Assume, on the contrary, that $\mu^+-\mu^->2\d$. Take the middle point of the gap $w=(\mu^-+\mu^+)/2$, so
\begin{equation*}
\begin{split}
w-\mu^- >\d \ \ &\Rightarrow \ \ w-u>\d, \quad w>u+\d\ge u_0, \\
\mu^+ -w >\d \ \ &\Rightarrow \ \ v-w>\d, \quad w<v-\d\le v_0,
\end{split}
\end{equation*}
and therefore, $w\in[u_0,v_0]=\s(T_0)$.

Pick a positive number $\tau$ so that
\begin{equation}\label{ta}
0<\tau<\frac{\mu^+-\mu^- -2\d}{2\d}\,, \quad \frac{\mu^+-\mu^-}2>(1+\tau)\d.
\end{equation}
We invoke the resolvent operators $R(z,T_0)=(T_0-z)^{-1}$, $R(z,T)=(T-z)^{-1}$, and write the equalities
\begin{equation*}
\begin{split}
R(z,T)-R(z,T_0) &=-R(z,T)(T-T_0)R(z,T_0), \\
R(z,T) &=\lp I-R(z,T)(T-T_0)\rp\,R(z,T_0).
\end{split}
\end{equation*}
Put $z=w+i\ep$, $\ep>0$ so, by \eqref{ta},
\begin{equation*}
\begin{split}
\|R(z,T)\|^{-1} &={\rm dist}(z,\s(T))\ge\sqrt{(1+\tau)^2\d^2+\ep^2}>(1+\tau)\d, \\
\|R(z,T)\| &< \frac1{(1+\tau)\d}\,.
\end{split}
\end{equation*}
Then
$$ \|R(z,T)(T-T_0)\|<\frac1{1+\tau}, $$
and the operator $I-R(z,T)(T-T_0)$ is invertible with
$$ \|\bigl(I-R(z,T)(T-T_0)\bigr)^{-1}\|\le \frac{1+\tau}{\tau}. $$
Consequently,
\begin{equation*}
\begin{split}
R(z,T_0) &=\bigl(I-R(z,T)(T-T_0)\bigr)^{-1}\,R(z,T), \\
\|R(z,T_0)\| &\le \|\bigl(I-R(z,T)(T-T_0)\bigr)^{-1}\|\,\|R(z,T)\|\le \frac1{\tau\d}\,.
\end{split}
\end{equation*}

On the other hand, $w\in\s(T_0)$ implies
$$ \|R(w+i\ep,T_0)\|=\ep^{-1}\to +\infty,  \qquad \ep\to 0+. $$
The contradiction completes the proof.
\end{proof}

\begin{proposition}\label{stbh}
Let $J$ be a bounded Jacobi matrix $\eqref{defjac}$, and $(\mu^-,\mu^+)$ be a gap in its spectrum. Then
$$ \mu^+-\mu^-\le 2(\o_b+2\o_a). $$
\end{proposition}
\begin{proof}
We apply Lemma \ref{perspe} for the case $T=J$, $T_0=J_{\a,\b}$, a constant Jacobi matrix with $\a$'s
along the main diagonal and $\b$'s along the off-diagonals, where $\a$ and $\b$ are suitable constants. Then
\begin{equation*}
J-J_{\a,\b} =
\begin{bmatrix}
\ddots & \ddots & \ddots &  & \\
 & a_{-1}-\a & b_0-\b & a_0-\a &  & \\
 & & a_0-\a & b_1-\b & a_1-a & & \\
 & &  & a_1-\a & b_2-\b & a_2-\a & \\
 & &  & & \ddots & \ddots & \ddots
\end{bmatrix},
\end{equation*}
and, as is well known (see, e.g., \cite[formula (1.3.29)]{simsz}),
$$ \|J-J_{\a,\b}\|\le \sup_j|b_j-\b|+2\sup_j|a_j-\a|. $$
Take $\inf_j b_j\le\b\le\sup_j b_j$, $\inf_j a_j\le\a\le\sup_j a_j$, so for each $k\in\bz$
\begin{equation*}
\begin{split}
|b_k-\b| &\le \max(b_k-\inf_j b_j, \sup_j b_j-b_k)\le\o_b, \\
|a_k-\a| &\le \max(a_k-\inf_j a_j, \sup_j a_j-a_k)\le\o_a.
\end{split}
\end{equation*}
It follows that $\|J-J_{\a,\b}\|\le \o_b+2\o_a$, and application of Lemma \ref{perspe} completes the proof.
\end{proof}

The lower bound \eqref{lobo} is a straightforward consequence of the latter result.

\begin{remark}
The Borg--Hochstadt theorem is known to hold for a wider class of reflectionless Jacobi matrices \cite[Corollary 8.6]{tesc}.
One may conjecture that its quantitative form, Theorem \ref{main}, remains valid in this setting (at least in the finite-gap
case) as well.
\end{remark}


\begin{thebibliography}{9999}



\bibitem{boer}
Borwein P., Erdelyi T., {\it Polynomials and polynomial inequalities}, Graduate Texts in Mathematics, {\bf 161}, Springer, 1995.

\bibitem{hoch75}
Hochstadt H., On the theory of Hill's matrices and related inverse spectral problems, Linear Alg. Appl., {\bf 11} (1975), 41--52.

\bibitem{hoch84}
Hochstadt H., An inverse spectral theorem for a Hill's matrix, Linear Alg. Appl., {\bf 57} (1984), 21--30.

\bibitem{kokr}
Korotyaev E., Krasovsky I., Spectral estimates for periodic Jacobi matrices, Comm. Math. Phys. {\bf 234} (2003), 517--532.

\bibitem{koku09}
Korotyaev E., Kutsenko A., Borg-type uniqueness theorems for periodic Jacobi opeerators with matrix-valued coefficients,
Proc. AMS, {\bf 137} (2009), 1989--1996.

\bibitem{kuku}
Kumar V. K., Kumar G. K., A pseudospectral analogue of discrete Borg-type theorems, preprint arXiv:1609.06518v1, 2016.

\bibitem{ma14}
Marchenko V. A., Hill discriminants and spectra of periodi Jacobi matrices, St. Petersburg Math. J. {\bf 25} (2014), no. 2, 265--269.

\bibitem{tesc}
Teschl G., {\it Jacobi Operators and Completely Integrable Nonlinear Lattices}, Math. Surveys and Monographs, {\bf 72}, AMS, Providence, RI, 2000.

\bibitem{simsz}
Simon B., {\it Szeg\H{o}'s Theorem and Its Descendants}, Spectral Theory for $L^2$ perturbations of Orthogonal Polynomials,
Princeton Unversity Press, Princeton and Oxford, 2011.



\end{thebibliography}
\end{document}